\documentclass[12pt]{article}
\usepackage[T1]{fontenc}

\usepackage{amssymb,amsthm,amsmath}

\usepackage{latexsym}
\usepackage{makeidx}
\usepackage{showidx}
\usepackage{graphicx}

\usepackage{fancybox}

\newfont{\bb}{msbm10 at 12pt}

\newtheorem{teorema}{Theorem}
\newtheorem{proposicion}{Proposition}
\newtheorem{lema}{Lemma}
\newtheorem{corolario}{Corollary}

\newtheorem{nota}{Remark}

\def\r{\hbox{\bb R}}

\def\fl{\longrightarrow}
\def\c{\hbox{\bb C}}

\def\z{\hbox{\bb Z}}
\def\Si{\Sigma}

\newcommand{\g}[2]{\langle #1,#2 \rangle }

\begin{document}

\begin{center}
\rule{15cm}{1.5pt}\vspace{1cm}

{\LARGE \bf Proper harmonic maps from hyperbolic\\[.5cm] Riemann surfaces into the Euclidean plane.}\\
\vspace{.5cm}

{\bf Antonio Alarcón$^a$ and José A. Gálvez${}^{b}$}\vspace{.5cm}

\rule{15cm}{1.5pt}
\end{center}

\noindent{\small ${}^a$Departamento de Matemática Aplicada, Facultad de Informática, Universidad de Murcia, 30100 Espinardo, Murcia, Spain; e-mail: ant.alarcon@um.es \\
${}^b$Departamento de Geometr\'\i a y Topolog\'\i a, Facultad de Ciencias,
Universidad de Granada, 18071 Granada, Spain; e-mail: jagalvez@ugr.es}

\begin{abstract}
Let $\Si$ be a compact Riemann surface and $D_1,\ldots,D_n$ a finite number of
pairwise disjoint closed disks of $\Si$. We prove the existence of a proper harmonic
map into the Euclidean plane from a hyperbolic domain $\Omega$ containing
$\Si\backslash\cup_{j=1}^n D_j$ and of its topological type. Here, $\Omega$ can be
chosen as close as necessary to $\Si\backslash\cup_{j=1}^n D_j$. In particular, we
obtain proper harmonic maps from the unit disk into the Euclidean plane, which
disproves a conjecture posed by R. Schoen and S.T. Yau.
\end{abstract}

\section{Introduction.}

It is classically known that there is no proper holomorphic map from the unit disk
into the complex plane. To this respect, in the more general class of harmonic maps,
E. Heinz proved, in 1952, there is no harmonic diffeomorphism from the unit disk
onto the Euclidean plane $\r^2$ \cite{H}. However, the problem of existence of a
proper harmonic map from the unit disk into the Euclidean plane remained open. Thus,
in 1985, R. Schoen and S.T. Yau conjectured the non-existence of these maps
\cite{SY}.

We solve that problem and, in fact, prove

\begin{teorema}\label{t1}
Let $\Sigma$ be a compact Riemann surface and $D_1,\ldots,D_n$ a family of pairwise
disjoint closed disks of $\Si$. Consider $n$ closed disks $D'_1,\ldots,D'_n$ such
that $D'_j\subseteq int D_j$, $j=1,\ldots, n$, then there exists a proper harmonic
map
$$
\varphi:\Omega\fl\r^2,
$$
where $\Omega$ is a domain of the topological type of
$\Si\backslash\left(\cup_{j=1}^nD_j\right)$ satisfying
\begin{equation}\label{contenido}
\Si\backslash\left(\cup_{j=1}^nD_j\right)\subseteq \Omega\subseteq
\Si\backslash\left(\cup_{j=1}^nD'_j\right).
\end{equation}
\end{teorema}

This result provides the existence of proper harmonic maps from a hyperbolic Riemann
surface of arbitrary finite topology into the Euclidean plane. In particular, when
$\Sigma$ is the Riemann sphere and $n=1$, Theorem \ref{t1} gives us the existence of
proper harmonic maps from the unit disk, and hence settles negatively the conjecture
above.

At this point we observe that, in general, we cannot fix the conformal structure of
the domain $\Omega$ where the harmonic map is defined. However, due to the control
of its topological type, this can be done when $\Si$ is the Riemann sphere. Thus,
one obtains

\begin{corolario}\label{t2}
Let $\Omega\subseteq\c$ be any n-connected domain. Then there exist proper harmonic
maps from $\Omega$ into the Euclidean plane.
\end{corolario}

The paper is organized as follows. In Section \ref{s2} we establish the notation
that will be used throughout this work. In addition, we prove that there is no
proper harmonic map from a Riemann surface of hyperbolic type into a complete flat
surface non-isometric to the Euclidean plane. This fact also motivates our study and
shows that the conjecture is true when the target space is a complete flat surface
non-isometric to $\r^2$.

We start Section \ref{phm} by proving Theorem \ref{t1}, which gives the existence of
proper harmonic maps from hyperbolic Riemann surfaces into the Euclidean plane. In
its proof, we show a slightly more general result asserting that, under certain
conditions, a harmonic map can be approximated by proper harmonic maps in any
compact set of its domain. Moreover, as a consequence of the Theorem above we also
prove Corollary \ref{t2}.

Finally, in Section \ref{ML} we prove the following Lemma, which is the basic tool
in order to prove the previous results.
\begin{lema}\label{mainlemma}
Let $\widetilde{\Si}$ be a compact Riemann surface and $K_0\subseteq\widetilde{\Si}$
a compact set. Consider two closed disks $D_1$, $D_2$ of
$\Si=\widetilde{\Si}\backslash int K$ such that $D_1\subseteq int D_2$ and $K_0\cap
D_2$ is empty. Let $F:\Si\backslash int D_1\fl\r^2$ be a harmonic map  and $r,R$ two
positive real numbers  such that
$$
r<|F(p)|<R,\qquad \mbox{for all }p\in D_2\backslash int D_1.
$$
Given three positive real numbers $\varepsilon_1,\varepsilon_2,\varepsilon_3$, there
exist a closed disk $D$ and a harmonic map $G:\Si\backslash int D\fl\r^2$
satisfying:
\begin{enumerate}
\item $D_1\subseteq int D\subseteq D\subseteq int D_2$,
\item $|F(p)-G(p)|<\varepsilon_1$, for all $p\in\Si\backslash int D_2$,
\item $R-\varepsilon_2<|G(p)|<R$, for any $p\in\partial D$,
\item $r-\varepsilon_3<|G(p)|$, for each $p\in D_2\backslash int D$.
\end{enumerate}
\end{lema}
This lemma provides, starting from a harmonic map $F$, a new harmonic map $G$
defined in a slightly smaller domain. In addition, $G$ can be chosen so that it is
very close to $F$ in a fixed compact domain $\Si\backslash int D_2$ and the norm of
$G$ can be made as large as needed on the boundary of its domain and is
well-controlled outside $\Si\backslash int D_2$.

This is the key step in order to construct a suitable sequence of harmonic maps
whose limit will prove Theorem \ref{t1}. In the proof of Lemma \ref{mainlemma} we
shall employ some ideas used in the study of different topics such as minimal
surfaces (see, for instance, \cite{M}, \cite{AFM}) or holomorphic embeddings from
Riemann surfaces into the complex 2-plane $\c^2$ (see \cite{FW}).

\section{Preliminaries.}\label{s2}

Let us consider the Euclidean plane $\r^2$ with linear coordinates $(x,y)$ endowed
with its usual induced metric $dx^2+dy^2$.

Let $\Si$ be a Riemann surface and $F:\Si\fl\r^2$ a map. Given an orthonormal basis
$S=\{e_1,e_2\}$, we shall denote
$$
F_{(1,S)}:=\langle F,e_1\rangle\quad\mbox{and}\quad F_{(2,S)}:=\langle F,e_2\rangle.
$$
Thus, $F$ is a harmonic map into $\r^2$ if and only if $F_{(j,S)}:\Si\fl\r$ are
harmonic functions, $j=1,2$.

In our construction process we shall use that, given a harmonic map $F:\Si\fl\r^2$,
a harmonic function $h:\Si\fl\r$ and an orthonormal basis $S=\{e_1,e_2\}$, then the
new map $G:\Si\fl\r^2$ defined in local coordinates as
$$
G_{(1,S)}=F_{(1,S)}+h\quad\mbox{and}\quad G_{(2,S)}=F_{(2,S)},
$$
is also harmonic. This new harmonic map $G$ can be considered as a deformation of
$F$ in the direction of $e_1$.

On the other hand, we shall introduce the following notation for subsets of a
Riemann surface $\Si$:
\begin{itemize}
\item Given $K\subseteq\Si$, we denote by $\Si_K$ the set $\Si\backslash int(K)$, where
$int$ stands for the interior of a set.
\item Let $D_1,D_2$ be two closed disks in $\Si$. We shall denote $D_1<D_2$ if $D_2\subseteq
int(D_1)$, or equivalently, if $\Si_{D_1}\subseteq int(\Si_{D_2})$.
\end{itemize}

We finish this Section by observing that if $\Si$ is a Riemann surface of hyperbolic
type and $M$ is a complete flat surface which is not isometric to the Euclidean
plane, then there exists no proper harmonic map from $\Si$ into $M$. This fact shows
a clear difference between $\r^2$ and the rest of complete flat surfaces.
\begin{proposicion}
Let $\Si$ be a Riemann surface of hyperbolic type and $M$ a complete flat surface
non-isometric to the Euclidean plane. Then there is no proper harmonic map from
$\Si$ into $M$.
\end{proposicion}
\begin{proof}
Since $M$ is a complete flat surface non-isometric to the Euclidean plane then $M$
must be a cylinder, a Moebius band, a torus or a Klein bottle. On the other hand,
assume that $\varphi$ is a proper harmonic map from $\Si$ into $M$. As $\Si$ is not
compact and $\varphi$ is proper then $M$ cannot be compact. That is, $M$ must be a
cylinder or a Moebius band.

Thus, if $M$ is a complete flat cylinder then $M$ is isometric to the Euclidean
plane $\r^2$ under the following identification:
$$
(x_1,y_1), (x_2,y_2) \mbox{ are identified if }x_1=x_2+k\,T,\ y_1=y_2,
$$
for a fixed real number $T>0$, and $k\in\z$.

Analogously, if $M$ is a complete flat Moebius band then $M$ is isometric to the
Euclidean plane $\r^2$ under the identification:
$$
(x_1,y_1), (x_2,y_2) \mbox{ are related if }x_1=x_2+k\,T,\ y_1=(-1)^ky_2,
$$
for a fixed real number $T>0$, and $k\in\z$.

In both cases, let us consider the function $h:M\fl\r$ given by the distance from a
point to the compact geodesic $\r\times\{0\}$ on $M$, that is, $h(x,y)=|y|$. Then
$\Si_0=(h\circ\varphi)^{-1}(0)$ is compact and $h\circ\varphi:\Si\fl\r$ is a proper
function which is smooth and harmonic on $\Si\backslash \Si_0$. Thus, using
\cite[Corollary 7.7]{Gry}, $\Si$ would be parabolic.

Therefore, there is no proper harmonic map from a hyperbolic Riemann surface into
$M$.
\end{proof}
\section{Existence of proper harmonic maps}\label{phm}
This Section is devoted to the construction of proper harmonic maps from hyperbolic
domains of a compact Riemann surface $\Si$ into the Euclidean plane and the study of
their properties.

Thus, we start by proving our main result about the existence of proper harmonic
maps from hyperbolic Riemann surfaces into $\r^2$. This proof fundamentally relies
on Lemma \ref{mainlemma} which will be showed in Section \ref{ML}.\vspace{3mm}

\noindent {\it Proof of Theorem \ref{t1}.} Throughout this proof we show a more
general result which will be used for our density result in Corollary \ref{cor2}.
That is, we prove that given any harmonic map
$$
F:\Si\backslash\left(\cup_{j=1}^n
int(D'_j)\right)\fl\r^2
$$
and a real number $\delta_0>0$, there exist an open domain $\Omega$ and a proper
harmonic map $\varphi:\Omega\fl\r^2$ such that
\begin{equation}
\label{delta} |F(p)-\varphi(p)|<\delta_0,\qquad \mbox{for all
}p\in\Si\backslash\left(\cup_{j=1}^n int(D_j)\right),
\end{equation}
where $\Omega$ is of the topological type of $\Si\backslash\left(\cup_{j=1}^n
D_j\right)$ and satisfies (\ref{contenido}).

For that, we shall prove that fixed $j_0$, there exist a domain $\Omega_{j_0}$ and a
proper harmonic map $\varphi_{j_0}:\Omega_{j_0}\fl\r^2$ such that
\begin{equation}
\label{dosprima} \left|\frac{1}{n}\
F(p)-\varphi_{j_0}(p)\right|<\delta:=\frac{\delta_0}{n},\qquad \mbox{for all
}p\in\Si\backslash\left(\cup_{j=1}^n int(D_j)\right),
\end{equation}
where
$$
\Si\backslash\left(int(D_{j_0})\cup(\cup_{j\neq
j_0}int(D'_j))\right)\subseteq \Omega_{j_0}\subseteq
\Si\backslash\left(\cup_{j=1}^nint(D'_j)\right)
$$
and $\Omega_{j_0}\cap int(D_{j_0})$ is an open annulus.

In such a case, the new harmonic map $\varphi=\varphi_1+\ldots+\varphi_n$ is well
defined in the open domain $\Omega=\cap_{j=1}^n \Omega_j$ and proper. Therefore,
$\varphi$ proves the Theorem and also inequality (\ref{delta}).

Up to a translation, we can assume the origin is not contained in the compact set
$F(\cup_{j=1}^nD_j\backslash\cup_{j=1}^n int(D'_j))$. Thus, once $j_0$ is fixed,
there exist two positive real numbers $r_0,R_0$ such that
$$
r_0<\frac{1}{n}\ |F(p)|<R_0,\qquad\mbox{ for all }p\in\Si_{D'_{j_0}}\cap D_{j_0}.
$$

Let us consider a decreasing sequence of positive real numbers $\eta_k$ such that
$\sum_{k\geq 1} \eta_k\leq1/2$ and denote $R_k=R_0+k$. Let us also call
$D_1^0=D'_{j_0}$, $D_2^0=D_{j_0}$, $F_0=(1/n)F$ and $\Sigma'=\Sigma\backslash
int(\cup_{j\neq j_0}D'_j)$. Our first goal is to prove that there exist two
sequences of closed disks $D_1^k, D_2^k$ and a sequence of harmonic maps
$F_k:\Si_{D_1^k}\fl\r^2$ satisfying
\begin{itemize}
\item[(p1)] $D_2^k<D_2^{k+1}<D_1^{k+1}<D_1^{k}$.
\item[(p2)] $|F_{k+1}(p)-F_k(p)|<\delta\,\eta_{k+1}$, for all $p\in\Si'_{D_2^k}$.
\item[(p3)] $R_{k+1}-\eta_{k+1}<|F_{k+1}(p)|<R_{k+1}$, for any $p\in\Si'_{D_1^{k+1}}\cap D_2^{k+1}$.
\item[(p4)] $R_k-\eta_k-\eta_{k+1}<|F_{k+1}(p)|$, for each $p\in\Si'_{D_1^{k+1}}\cap D_2^k$.
\end{itemize}

Using Lemma \ref{mainlemma} for the Riemann surface $\Sigma$, the compact domain
$\cup_{j\neq j_0}D'_j$ and the harmonic map $F_0$, we have the existence of a closed
disk $D_1^1$ such that
\begin{itemize}
\item $D_2^0<D_1^{1}<D_1^{0}$.
\item $|F_{1}(p)-F_0(p)|<\delta\,\eta_{1}$, for all $p\in\Si'_{D_2^0}$.
\item $(R_0+1)-\eta_1<|F_1(p)|<R_0+1$, for any $p\in\partial D_1^1$.
\end{itemize}
Thus, from the last inequality, one gets the existence of a closed disk $D_2^1$ such
that $D_2^0<D_2^{1}<D_1^{1}$ and
$$
(R+1)-\eta_1<|F_1(p)|<R+1,\quad \mbox{ for any } p\in\Si'_{D_1^1}\cap D_2^1.
$$
Therefore, the properties (p1),(p2) and (p3) are satisfied. Observe that the
property (p4) has no sense for this first step.

Now, let us assume the sequences $D_1^{k}, D_2^{k}, F_{k}$ are well defined until
$k=m$ and satisfy properties (p1)--(p4). Then we define the following elements of
the sequences for $k=m+1$ as follows. Note that property (p4) will not be used in
the construction of the remaining terms of the sequence.

Using again Lemma \ref{mainlemma} and bearing in mind that (p3) is satisfied for
$k=m$, we obtain the existence of a closed disk $D_1^{m+1}$ such that
\begin{itemize}
\item $D_2^m<D_1^{m+1}<D_1^{m}$.
\item $|F_{m+1}(p)-F_m(p)|<\delta\,\eta_{m+1}$, for all $p\in\Si'_{D_2^m}$.
\item $R_{m+1}-\eta_{m+1}<|F_{m+1}(p)|<R_{m+1}$, for any $p\in\partial D_1^{m+1}$.
\item $(R_m-\eta_m)-\eta_{m+1}<|F_{m+1}(p)|$, for each
$p\in\Si'_{D_1^{m+1}}\cap D_2^m.$
\end{itemize}
So, since $D_2^m<D_1^{m+1}$ and from the inequality satisfied for the points in
$\partial D_1^{m+1}$, one obtains the existence of a closed disk $D_2^{m+1}$ such
that $D_2^m<D_2^{m+1}<D_1^{m+1}$ and
$$
R_{m+1}-\eta_{m+1}<|F_{m+1}(p)|<R_{m+1},\quad \mbox{ for any }
p\in\Si'_{D_1^{m+1}}\cap D_2^{m+1}.
$$
Hence, the properties (p1)--(p4) are also satisfied for $k=m+1$.

Let us now define the desired harmonic map $\varphi_{j_0}$. For that, we first
define $\Omega_{j_0}=\cup_{k\geq 0}\Si'\backslash D_2^k$. Observe that
$\Si'\backslash D_2^k\subseteq \Si'\backslash D_2^{k+1}$ and $int(D_2^k)\backslash
D_2^{k+1}$ is an open annulus. So, an elementary topological argument gives us
$\Omega_{j_0}\cap int(D_2^0)$ is an open annulus.

In addition, let us see that for any compact set $K\subseteq\Omega_{j_0}$ we have
that ${F_k}_{|K}$ is a Cauchy sequence of harmonic maps. In such a case, the limit
map $\varphi_{j_0}$ is harmonic and well defined in $\Omega_{j_0}$.

Thus, let us fix the compact set $K\subseteq\Omega_{j_0}$. Then, there exists a
non-negative entire number $k_0$ such that $K\subseteq\Si'_{D_2^{k_0}}$. Hence,
using (p2), for $k_1>k_2\geq k_0$
\begin{equation}
\label{dens}
|F_{k_1}(p)-F_{k_2}(p)|<\delta\,\sum_{j=k_2+1}^{k_1}\eta_j,\qquad\mbox{for all }p\in
K,
\end{equation}
 and, so, ${F_k}_{|K}$ is a Cauchy sequence since
 $\sum_{j\geq1}\eta_j\leq1/2<\infty$.

Taking limits in (\ref{dens}), we also obtain that $(\ref{dosprima})$ is satisfied,
since if $p\in\Si'_{D_2^0}$ then
$$
|F_{0}(p)-\varphi_{j_0}(p)|\leq\delta\,\sum_{k\geq1}\eta_k<\delta.
$$

Finally, we prove that $\varphi_{j_0}$ is proper. Let $m$ be a positive entire and
consider a point $p\in\Si'_{D_2^{m+1}}\cap D_2^m$, then  from (\ref{dens}) and (p4)
one has
\begin{eqnarray*}
&&|\varphi_{j_0}(p)-F_{m+1}(p)|\leq \delta\sum_{k\geq m+2}\eta_k<\delta,\\
&&|F_{m+1}(p)|>R_m-\eta_m-\eta_{m+1}>R_m-1=R_0+m-1.
\end{eqnarray*}
And so $|\varphi_{j_0}(p)|>m+(R_0-\delta-1)$.

In particular, if $q\in\Omega_{j_0}\cap D_2^m$ then there exists $m'\geq m$ such
that $q\in\Si'_{D_2^{m'+1}}\cap D_2^{m'}$. Thereby,
$$
|\varphi_{j_0}(q)|>m+(R_0-\delta-1),\qquad\mbox{ for all }q\in\Omega_{j_0}\cap
D_2^m.
$$

Hence, for any real number $N>0$ there exists a positive entire $m$  such that the
set $\{q\in\Omega_{j_0}:\ |\varphi_{j_0}(q)|\leq N\}$ is contained in the compact
set $\Si'_{D_2^m}$. Therefore, $\varphi_{j_0}$ is proper as we wanted to
show.\hfill$\square$

When we consider $\Si$ as the Riemann sphere in Theorem \ref{t1}, and remove only
one closed disk, we get proper harmonic maps from a hyperbolic simply connected
domain $\Omega$ into $\r^2$. Thus, $\Omega$ must be conformal to the unit disk,
obtaining proper harmonic maps from the unit disk into the Euclidean plane.

\begin{nota}
In general, the proper harmonic maps given by Theorem \ref{t1} could be
non-surjective. However, though we shall not worry about this problem, following the
Proof of Lemma \ref{mainlemma}, it is always possible to find a proper harmonic map
$\varphi$ given by Theorem \ref{t1} which is onto.
\end{nota}
It is important to note that in Theorem \ref{t1} we control the topological type of
the domain $\Omega$ where the harmonic map is defined. However, we cannot fix the
conformal type in general. But, this can be done when we start with a domain in the
Riemann sphere.

So, as an immediate consequence of the previous Theorem, we prove that if $\Omega$
is an n-connected domain in $\c$, then there exist proper harmonic maps from
$\Omega$ into $\r^2$.\vspace{.3cm}

\noindent {\it Proof of Corollary \ref{t2}.} Let $\Omega$ be an n-connected domain
in $\c$.

If $\c\backslash\Omega$ is a finite number of points, the result is obvious.
Otherwise, $\Omega$ is conformal to a bounded domain whose boundary is made by a
finite number of pairwise disjoint circles $C_1,\ldots,C_k$ and a finite number of
points $p_{k+1},\ldots,p_{n}$.

Let $D_j$ be the closed disk in the Riemann sphere
$\overline{\c}:=\c\cup\{\infty\}$, such that $D_j$ does not intersect $\Omega$ and
whose boundary is $C_j$. Since, $\overline{\c}\backslash D_j$ is conformal to the
unit disk, using Theorem \ref{t1}, there exists a proper harmonic map
$\varphi_j:\overline{\c}\backslash D_j\fl\r^2$.

In addition, for the points $p_j$ we take $\varphi_j$ as a holomorphic function from
$\overline{\c}\backslash\{p_j\}$ with a pole at $p_j$.

Thus, it is easy to see that the new harmonic map $\varphi=\sum_{j=1}^n
\varphi_j:\Omega\fl\r^2$ is proper.\hfill$\square$\vspace{.3cm}

On the other hand, observe from (\ref{delta}) that every harmonic map defined on a
Riemann surface minus $n$ open disks into $\r^2$ can be approximated by proper
harmonic maps in every compact set contained in the interior of the domain. That is,
we have really showed the following density result.

\begin{corolario}\label{cor2}
Let $\Si$ be a Riemann surface and $D_1,\ldots,D_n$ a family of pairwise disjoint
closed topological disks of $\Si$. Consider a harmonic map
$F:\Si\backslash\cup_{j=1}^nint(D_j)\fl\r^2$. Then for any compact set $K\subseteq
\Si\backslash\cup_{j=1}^n D_j$ and $\delta_0>0$, there exist a hyperbolic domain
$\Omega\subseteq \Si$ containing $K$ and a proper harmonic map
$\varphi:\Omega\fl\r^2$ such that
$$
|\varphi(p)-F(p)|<\delta_0,\qquad\mbox{for all }p\in K.
$$
\end{corolario}

Notice Schoen and Yau related their conjecture with the problem of non-existence of
a hyperbolic minimal surface in $\r^3$ which properly projects into a plane
\cite{SY}. However, since there exist proper harmonic maps from hyperbolic Riemann
surfaces, the problem of finding such a hyperbolic minimal surface remains open.

Moreover, Schoen and Yau proved that there is no harmonic diffeomorphism from the
unit disk onto a complete surface with non-negative Gauss curvature \cite{SY},
extending the result of Heinz for the Euclidean plane \cite{H}. To this respect, it
would be interesting to study the existence of proper harmonic maps from hyperbolic
Riemann surfaces into a complete surface with non-negative Gauss curvature.

Finally, it should be also mentioned that other Picard type problems for harmonic
diffeomorphisms from the complex plane on Hadamard surfaces have been recently
studied \cite{CR}, \cite{GR}, though the used techniques in that context are
completely different.

\section{Proof of Lemma \ref{mainlemma}.}\label{ML}

We fix a real number $\varepsilon_0>0$ which satisfies some inequalities in terms of
the initial data $r,R,\varepsilon_1,\varepsilon_3$ and whose dependence will be made
clear along this proof.

We take a closed disk $D_3$ on $\Si$ such that $D_2<D_3<D_1$ and a set of points
$\{p_1,\ldots,p_n\}$ in the open set $int(\Si_{D_3})\backslash \Si_{D_2}$
satisfying:
\begin{itemize}
\item[(a1)] There exists a closed disk $D_4$ such that $\{p_1,\ldots,p_n\}\subseteq\partial D_4$ with
$D_2<D_4<D_3$. Moreover, we can assume these points are ordered along the circle
$\partial D_4$.
\item[(a2)] There exist open disks $B_j\subseteq int(\Si_{D_3})\backslash \Si_{D_2}$
such that $p_j,p_{j+1}\in B_j$ and
$$
|F(p)-F(q)|<\varepsilon_0,\qquad \mbox{for each
} p,q\in B_j.
$$
Here and from now on, we shall suppose  $p_{n+1}=p_1$.
\item[(a3)] The well oriented orthonormal basis
$S_j=\{\frac{F(p_j)}{|F(p_j)|},i\,\frac{F(p_j)}{|F(p_j)|}\}$ satisfy
$$
\left|\frac{F(p_j)}{|F(p_j)|}-\frac{F(p_{j+1})}{|F(p_{j+1})|}\right|<\frac{\varepsilon_0}{3(R-r)},\qquad
j=1,\ldots,n.
$$
\end{itemize}

Observe that these points can be chosen by continuity of $F$ in the compact set
$\Si_{D_1}$.

Now, let us take a Riemannian metric $\g{}{}$ on $\Si$ which is compatible with the
conformal structure and consider a real number $\delta>0$ such that:
\begin{itemize}
\item $D_5=D_4\cup\left(\cup_{j=1}^n D(p_j,\delta)\right)$ is a topological
closed disk. Here, $D(p,\delta)$ denotes the closed geodesic disk center at $p$ and
radius $\delta$ for the metric $\g{}{}$.
\item $D(p_j,\delta)\cup D(p_{j+1},\delta)\subseteq B_j$, for all $j=1,\ldots,n.$
\item $D(p_j,\delta)\cap D(p_{k},\delta)$ is empty for any $j\neq k.$
\item $F_{(1,S_j)} (p)>r$ for each $p\in D(p_j,\delta)$.
\end{itemize}

Let $\widehat{\zeta}_j:\Si_{D_1}\backslash\{p_j\}\fl\c$ be a holomorphic function
with a simple pole at $p_j$. Its existence is clear if $\Si$ is a topological sphere
and assured by the Noether gap Theorem when the genus of $\Si$ is positive.

Up to multiplication of $\widehat{\zeta}_j$ by a complex number, if necessary, there
exist a point $q_j\in\partial D(p_j,\delta)\cap int(\Si_{D_4})$ and a simple curve
$\beta_j:[0,1]\fl D(p_j,\delta)$ such that
\begin{itemize}
\item $\beta_j(0)=q_j$ and $\beta_j(1)=p_j$.
\item $\beta_j(]0,1])\subseteq int D(p_j,\delta)$.
\item $\widehat{\zeta}_j(\beta_j(t))$ is a positive real number for all $t\in[0,1[$.
\item $\lim_{t\rightarrow 1}\widehat{\zeta}_j(\beta_j(t))=\infty.$
\end{itemize}

Let us denote by $K_j$ to the compact set $\Si_{D_1}\backslash int(D(p_j,\delta))$
and by $\zeta_j$ to the real part of $\widehat{\zeta}_j$. Then, we can chose
$\kappa>0$ small enough in order to satisfy
\begin{equation}
\label{4} \kappa\ {\rm max}\{|\zeta_j(p)|:\ p\in K_j\}<{\rm
min}\left\{\frac{\varepsilon_0}{n},3\,(R-r)\right\},\qquad j=1,\ldots,n.
\end{equation}

And we call $a_j=\beta_j(t_j)$ where
$$
t_j={\rm min}\{t\in]0,1[:\ \kappa\ \widehat{\zeta}_j(\beta_j(t))=\kappa\
\zeta_j(\beta_j(t))=3(R-r)\}.
$$

Now, we construct a family of harmonic maps $h_j:U_j\fl\r^2$, where
$U_j=\Si_{D_1}\backslash\{p_1,\ldots,p_j\}$. These harmonic maps are defined as
follows
\begin{itemize}
\item $h_0=F$.
\item In the basis $S_j$ we have
$$
{h_j}_{(1,S_j)}={h_{j-1}}_{(1,S_j)}+\kappa\,\zeta_j,\qquad
{h_j}_{(2,S_j)}={h_{j-1}}_{(2,S_j)}.
$$
\end{itemize}
\begin{lema}
These new harmonic maps satisfy the following properties
\begin{itemize}
\item[{\rm (b1)}] $|h_j(p)-h_{j-1}(p)|<\varepsilon_0/n$, for any $p\in U_j\backslash
int(D(p_j,\delta))$.
\item[{\rm (b2)}] $|h_j(a_j)-h_{j-1}(a_{j-1})|<4\,\varepsilon_0$.
\item[{\rm (b3)}] $|h_n(a_j)-h_n(a_{j+1})|<6\,\varepsilon_0$.
\item[{\rm (b4)}] $|h_n(a_j)|>2R-r$.
\end{itemize}
\end{lema}
\begin{proof}
In order to obtain the inequality (b1), we have from (\ref{4})
$$
|h_j(p)-h_{j-1}(p)|=|\kappa\,\zeta_j(p)|<\varepsilon_0/n.
$$

Let us write $S_j=\{e_1^j,e_2^j\}$, that is, $e_1^j=F(p_j)/|F(p_j)|$ and
$e_2^j=i\,F(p_j)/|F(p_j)|$, then one has for $j\geq 2$
\begin{eqnarray*}
|h_j(a_j)-h_{j-1}(a_{j-1})|&=&|h_{j-1}(a_j)+\kappa\,\zeta_j(a_j)e_1^j-h_{j-2}(a_{j-1})-
\kappa\,\zeta_{j-1}(a_{j-1})e_1^{j-1}|\\
&\leq&|h_{j-1}(a_j)-F(a_j)|+|F(a_j)-F(a_{j-1})|\\&&+|F(a_{j-1})-h_{j-2}(a_{j-1})|+3(R-r)|e_1^j-e_1^{j-1}|.
\end{eqnarray*}

In addition, from (b1), $|h_{j-1}(a_j)-F(a_j)|$ and $|h_{j-2}(a_{j-1})-F(a_{j-1})|$
are less than $\varepsilon_0$. On the other hand, using (a2),
$|F(a_j)-F(a_{j-1})|<\varepsilon_0$ since $a_{j-1},a_j\in B_{j-1}$. Hence, from
(a3), the inequality (b2) is proved.

In order to obtain (b3) we observe
\begin{eqnarray*}
|h_n(a_j)-h_n(a_{j+1})|&\leq& |h_n(a_j)-h_j(a_{j})|+|h_j(a_j)-h_{j+1}(a_{j+1})|\\&&+
|h_{j+1}(a_{j+1})-h_n(a_{j+1})|\ <\ \varepsilon_0+4\varepsilon_0+\varepsilon_0\ =\
6\varepsilon_0,
\end{eqnarray*}
where we have used (b1) and (b2).

Moreover, from (b1),
\begin{eqnarray*}
|h_n(a_j)|&>&|h_j(a_j)|-\varepsilon_0\\
&=&\sqrt{\left({h_{j-1}}_{(1,S_j)}(a_j)+\kappa\,\zeta_j(a_j)\right)^2+{h_{j-1}}_{(2,S_j)}(a_j)^2}-\varepsilon_0\\
&\geq&|{h_{j-1}}_{(1,S_j)}(a_j)+3(R-r)|-\varepsilon_0,
\end{eqnarray*}
since $\kappa\,\zeta_j(a_j)=3(R-r)$.

We also have from (a2) and (b1)
\begin{eqnarray*}
|{h_{j-1}}_{(1,S_j)}(a_j)-F_{(1,S_j)}(p_j)|&\leq&|{h_{j-1}}(a_j)-F(p_j)|\\
&\leq& |{h_{j-1}}(a_j)-F(a_j)|+|F(a_j)-F(p_j)|\ <\ 2\varepsilon_0.
\end{eqnarray*}

As $F_{(1,S_j)}(p_j)=|F(p_j)|>r$ by hypothesis, then
${h_{j-1}}_{(1,S_j)}(a_j)>r-2\varepsilon_0$ and so
$$
|h_n(a_j)|>3R-2r-3\varepsilon_0>2R-r.
$$
Here, we use that $\varepsilon_0$ can be initially chosen with
$R-r-3\varepsilon_0>0$.
\end{proof}

Let $\widetilde{C}_j$ be a simple regular curve contained in $int(D(p_j,\delta))$
which is transversal to $\beta_j([0,1])$ at $a_j$. Consider a small closed connected
neighborhood $C_j$ of the curve $\widetilde{C}_j$ at the point $a_j$, such that each
connected component of $C_j\backslash\{a_j\}$ lies in one side of $\beta_j([0,1])$.
Moreover, $C_j$ is chosen small enough in order to satisfy
\begin{itemize}
\item[(b5)] $|h_j(p)-h_j(a_j)|<\varepsilon_0$ for all $p\in C_j$.
\end{itemize}

Consider a topological closed strip $G_j$ contained in $D(p_j,\delta)$ such that
$\beta_j(]0,t_j[)\subseteq int(G_j)$, and bounded by $C_j$, $\partial D(p_j,\delta)$
and two closed simple curves lying at each side of $\beta_j([0,t_j])$ and joining a
point of $\partial D(p_j,\delta)$ with a point of $C_j$ (see Figure 1).

In addition, we take $G_j$ small enough such that
\begin{itemize}
\item[(b6)] $|h_j(p)-h_{j-1}(p)-k\,\zeta_j(p)e_1^j|<\varepsilon_0$ for all $p\in G_j$.
\item[(b7)] $0<k\,\zeta_j(p)<3(R-r)+\varepsilon_0$ for each $p\in G_j$.
\end{itemize}
Note that $G_j$ can be chosen in order to satisfy these two properties. For (b6), it
is sufficient to observe that $h_j(p)=h_{j-1}(p)+k\,\zeta_j(p)e_1^j$ for any point
$p$ on $\beta_j([0,1[)$. And (b7) is a consequence of the definition of $a_j$ and
$\zeta_j$.

Now, we consider the compact set
$$
K=\left(\Si_{D_4}\backslash\cup_{j=1}^n int(D(p_j,\delta))\right)\cup
\left(\cup_{j=1}^n G_j\right)
$$
and $D_6$ the closed disk $\Sigma\backslash int K$. That is, $K=\Si_{D_6}$.

For this new set we observe that
\begin{itemize}
\item $\beta_j([0,t_j])\subseteq\Si_{D_6}$,
\item $D_2<D_6$,
\item the harmonic maps $h_j$ are well defined in $\Si_{D_6}$,
\item $\Si_{D_6}\cap int (D(p_j,\delta))\subseteq G_j\subseteq \Si_{D_6}\cap
D(p_j,\delta)$.
\end{itemize}

Finally, we take a closed disk $D_7$ such that $D_6<D_7$ and the harmonic maps $h_j$
are well defined in $\Si_{D_7}$.

Until now, we have deformed the initial harmonic map $F$ obtaining a new harmonic
map $h_n$. These harmonic maps are close in the compact domain $\Si_{D_2}$ (property
(b1)). And the norm of $h_n$ is bigger than $2R-r>R$ for the family of points
$a_1,\ldots,a_n$ (property (b4)). Thus, we have obtained a harmonic map which
satisfies part of the conditions of Lemma \ref{mainlemma}, but only near the points
$a_j$.

Our goal is to deform $h_n$ in order to obtain the harmonic map which proves the
Lemma. For that, we shall construct, from $h_n$, a sequence of $n$ harmonic maps.
Each harmonic map will be close to the previous one in a large compact domain and
only modified along a neighborhood of a curve joining $a_{j-1}$ and $a_j$. The last
harmonic map of this sequence will be the one which solves Lemma \ref{mainlemma}.

\begin{figure}[h]
\mbox{}
\begin{center}
\includegraphics[clip,width=15cm]{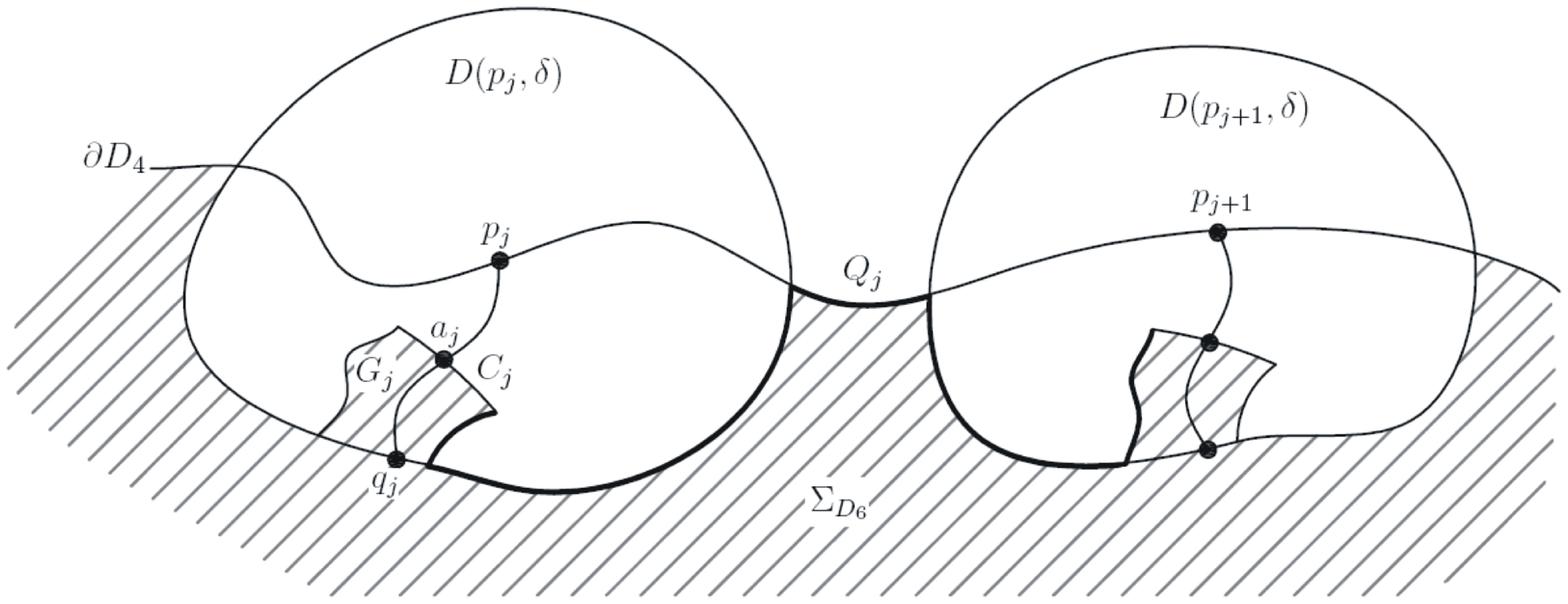}\\
\mbox{ }\hfill Figure 1.\hfill\mbox{ }
\end{center}
\end{figure}

Let us denote by $Q_j$ the connected component of $\overline{\partial
D_6\backslash(C_j\cup C_{j+1})}$ which does not intersect $C_k$ for all $k\neq
j,j+1$. These closed curves satisfy (see Figure 1)
\begin{itemize}
\item $Q_j\cap Q_k$ is empty for each $j\neq k$,
\item $Q_j\subseteq B_j$,
\item $Q_j\cap D(p_k,\delta)$ is empty for each $k\neq j,j+1$,
\item $\partial D_6=\left(\cup_{j=1}^n Q_j\right)\cup\left(\cup_{j=1}^n C_j\right)$.
\end{itemize}

Consider an open neighborhood $\widehat{C}_j$ of $C_j$ such that
\begin{equation}
\label{(f1)} |h_n(p)-h_n(a_j)|<3\,\varepsilon_0,\qquad \mbox{ for all }
p\in\widehat{C}_j\cap\Si_{D_5}.
\end{equation}
This can be done by continuity, because for any $p\in C_j$
$$
|h_n(p)-h_n(a_j)|\leq
|h_n(p)-h_j(p)|+|h_j(p)-h_j(a_j)|+|h_j(a_j)-h_n(a_j)|<3\,\varepsilon_0,
$$
where we have used (b1) and (b5).

We also consider the neighborhoods of $Q_j$
$$
Q_j(L)=\{p\in\Si_{D_1}:\ {\rm dist}(p,Q_j)\leq L\},
$$
and fix $L_0>0$ small enough in order to satisfy similar properties to $Q_j$, that
is,
\begin{itemize}
\item $Q_j(L_0)\cap Q_k(L_0)$ is empty for any $j\neq k$,
\item $Q_j(L_0)\subseteq B_j\cap int(\Si_{D_7})$,
\item $Q_j(L_0)\cap D(p_k,\delta)$ is empty for each $k\neq j,j+1$.
\end{itemize}

We are now ready to define a sequence of harmonic maps $g_j:\Si\backslash
D_7\fl\r^2$, $j=0,\ldots,n$. For that, we choose some basis which will indicate the
direction of the deformation. So, let
$\overline{S}_j=\{\overline{e}_j^1,\overline{e}_j^2\}$, with
$$
\overline{e}_j^2=\frac{h_n(a_j)}{|h_n(a_j)|},\quad
\overline{e}_j^1=-i\,\overline{e}_j^2.
$$

Let $g_0=h_n$ and let us define $g_{j}$ from $g_{j-1}$ as follows. Let
$$
\tau_j=2R-r-{\rm min}\left\{{g_{j-1}}_{(1,\overline{S}_j)}(p):\ p\in Q_j\right\}.
$$
By Royden's Theorem \cite[Theorem 10]{R}, there exists a holomorphic function
$\widehat{\xi}_j:\Sigma_{D_7}\fl\c$ such that
\begin{itemize}
\item $|\widehat{\xi}_j(p)|<\varepsilon_0/n$, for any $p\in \Si_{D_6}\backslash
int(Q_j(L_0))$,
\item $|\widehat{\xi}_j(p)-\tau_j|<\varepsilon_0/n$, for all $p\in
Q_j(L_0/2)$.
\end{itemize}
If we take the harmonic function $\xi_j:\Si_{D_7}\fl\r$ given by the real part of
$\widehat{\xi}_j$, then
\begin{itemize}
\item $|\xi_j(p)|<\varepsilon_0/n$, for any $p\in \Si_{D_6}\backslash
int(Q_j(L_0))$,
\item $|\xi_j(p)-\tau_j|<\varepsilon_0/n$, for all $p\in Q_j(L_0/2)$.
\end{itemize}
Thus, we define
$$
{g_j}_{(1,\overline{S}_j)}={g_{j-1}}_{(1,\overline{S}_j)}+\xi_j,\qquad
{g_j}_{(2,\overline{S}_j)}={g_{j-1}}_{(2,\overline{S}_j)}.
$$

The harmonic maps $g_j$ satisfy the following properties
\begin{itemize}
\item[(c1)] $|g_j(p)-g_{j-1}(p)|<\varepsilon_0/n$, for any $p\in \Si_{D_6}\backslash
int(Q_j(L_0))$.
\item[(c2)]
$|g_j(p)-g_{j-1}(p)-\tau_j\,\overline{e}_1^j|=|\xi_j(p)-\tau_j|<\varepsilon_0/n$,
for all $p\in Q_j(L_0/2)$.
\end{itemize}

Finally, we define the harmonic map $G:\Si_{D_6}\fl\r^2$ as $G=g_n$. Therefore,
Lemma \ref{mainlemma} will be a consequence of the next result by shrinking the
domain $\Si_{D_6}$.

\begin{lema}\label{l4}
The harmonic map $G$ satisfies
\begin{enumerate}
\item $|F(p)-G(p)|<\varepsilon_1$, for all $p\in\Si_{D_2}$,
\item $|G(p)|>R$, for any $p\in\partial D_6$,
\item $|G(p)|>r-\varepsilon_3$, for each $p\in \Si_{D_6}\cap D_2$.
\end{enumerate}
\end{lema}
\begin{proof}
The first inequality is a consequence of (b1) and (c1), since if $p\in\Si_{D_2}$
$$
|F(p)-G(p)|\leq |g_n(p)-g_0(p)|+|h_n(p)-h_0(p)|<2\,\varepsilon_0,
$$
where we can choose initially $\varepsilon_0<\varepsilon_1/2$.

In order to prove the second inequality, we distinguish four cases:\\[3mm]
(i) Let $p\in \widehat{C}_j\cap Q_j(L_0)$. Using (c1), (\ref{(f1)}) and (b4), one
gets
\begin{eqnarray*}
|G(p)|&=&|g_n(p)|\ >\ |g_j(p)|-\varepsilon_0\ \geq\
|{g_j}_{(2,\overline{S}_j)}(p)|-\varepsilon_0\ =\
|{g_{j-1}}_{(2,\overline{S}_j)}(p)|-\varepsilon_0\\
&\geq&|{g_0}_{(2,\overline{S}_j)}(p)|-2\varepsilon_0\ =\
|{h_n}_{(2,\overline{S}_j)}(p)|-2\varepsilon_0\ >\
|{h_n}_{(2,\overline{S}_j)}(a_j)|-5\varepsilon_0\\
&=&|h_n(a_j)|-5\varepsilon_0\ >\ 2R-r-5\varepsilon_0>R.
\end{eqnarray*}
Where $\varepsilon_0$ initially satisfies $R-r-5\varepsilon_0>0$.\\[3mm]
(ii) Let $p\in \widehat{C}_j\cap Q_{j-1}(L_0)$. Reasoning as in the previous case
and using (b3)
\begin{eqnarray*}
|G(p)|&>&|{h_n}_{(2,\overline{S}_{j-1})}(a_{j})|-5\varepsilon_0\ >\
|{h_n}_{(2,\overline{S}_{j-1})}(a_{j-1})|-11\varepsilon_0\ =\ |h_n(a_{j-1})|-11\varepsilon_0\\
&>& 2R-r-11\varepsilon_0\ >\ R,
\end{eqnarray*}
with $R-r-11\varepsilon_0>0$  a priori.\\[3mm]
(iii) Let $p\in \widehat{C}_j\backslash \left(\cup_{k=1}^nQ_k(L_0)\right)$. From
(c1) and (\ref{(f1)}), one obtains
$$
|G(p)|\ >\ |g_0(p)|-\varepsilon_0\ =\ |h_n(p)|-\varepsilon_0\ >\
|h_n(a_j)|-4\varepsilon_0\ >\ 2R-r-4\varepsilon_0\ >\ R.
$$
(iv) Let $p\in Q_j\backslash\left(\cup_{k=1}^n \widehat{C}_k\right)$. Then, from
(c1) and (c2)
\begin{eqnarray*}
|G(p)|&>&|g_j(p)|-\varepsilon_0\ \geq\
|{g_j}_{(1,\overline{S}_j)}(p)|-\varepsilon_0\ >\
|{g_{j-1}}_{(1,\overline{S}_j)}(p)+\tau_j|-2\varepsilon_0\\
&\geq&{g_{j-1}}_{(1,\overline{S}_j)}(p)+\tau_j-2\varepsilon_0\ \geq\
2R-r-2\varepsilon_0>R.
\end{eqnarray*}

Finally, we prove the third inequality, that is, $|G(p)|>r-\varepsilon_3$. For that,
we need to distinguish five different cases, depending on the region of
$\Sigma_{D_6}$ where $p$ is
located.\\[3mm]
(i) Let $p\not\in\left(\cup_{j=1}^nD(p_j,\delta)\right)\cup\left(\cup_{j=1}^n
Q_j(L_0)\right)$. Using (c1) and (b1)
$$
|G(p)|\ >\ |h_n(p)|-\varepsilon_0\ >\ |F(p)|-2\varepsilon_0\ >\ r-2\varepsilon_0\ >\
r-\varepsilon_3,
$$
where initially $\varepsilon_0<\varepsilon_3/2$.\\[3mm]
(ii) Let $p\in D(p_j,\delta)\backslash \left(\cup_{k=1}^n Q_k(L_0)\right)$. In this
case we shall use (c1), (b1), (b6), (a2), (b7) and $\varepsilon_0$ is chosen less
than $\varepsilon_3/5$
\begin{eqnarray*}
|G(p)|&>&|h_n(p)|-\varepsilon_0\ >\ |h_j(p)|-2\varepsilon_0\ >\
|h_{j-1}(p)+\kappa\,\zeta_j(p)e_1^j|-3\varepsilon_0\\
&\geq&|{h_{j-1}}_{(1,S_j)}(p)+\kappa\,\zeta_j(p)|-3\varepsilon_0\ >\
|F_{(1,S_j)}(p)+\kappa\,\zeta_j(p)|-4\varepsilon_0\\
&\geq&|F_{(1,S_j)}(p_j)+\kappa\,\zeta_j(p)|-5\varepsilon_0\ =\
||F(p_j)|+\kappa\,\zeta_j(p)|-5\varepsilon_0\\
&>&r-5\varepsilon_0\ >\ r-\varepsilon_3.
\end{eqnarray*}
(iii) Let $p\in int(D(p_j,\delta))\cap Q_j(L_0)$. First, we observe
\begin{eqnarray*}
|h_n(a_j)-(F(p_j)+3(R-r)e_1^j)|&\leq&|h_n(a_j)-h_j(a_j)|+|F(a_j)-F(p_j)|\\
&+&|h_j(a_j)-F(a_j)-3(R-r)e_1^j|\ <\ 3\varepsilon_0,
\end{eqnarray*}
where we have used that $h_j(a_j)=h_{j-1}(a_j)+3(R-r)e_1^j$, and the formulas (a2)
and (b1).

In addition, we have  $|h_n(a_j)|>2R-r>R$ from (b4), and $|F(p_j)|+3(R-r)>3R-2r>R$.

Hence, as $F(p_j)+3(R-r)e_1^j=(|F(p_j)|+3(R-r))e_1^j$, one has
\begin{equation}\label{epsilon}
\begin{array}
{rcl} \left|\overline{e}_2^j-e_1^j\right|&=&{\displaystyle \left|
\frac{h_n(a_j)}{|h_n(a_j)|}-e_1^j\right|\ <\ 6\varepsilon_0{\rm min}\left\{
\frac{1}{|h_n(a_j)|},\frac{1}{|F(p_j)|+3(R-r)}\right\} }\\[4mm] &<& {\displaystyle
\frac{6\varepsilon_0}{R}.}
\end{array}
\end{equation}

Now, for the point $p$, we get
\begin{eqnarray*}
|G(p)|&>&|g_j(p)|-\varepsilon_0\ \geq\
|{g_j}_{(2,\overline{S}_j)}(p)|-\varepsilon_0\ >\
|{h_n}_{(2,\overline{S}_j)}(p)|-2\varepsilon_0\\
&>&|{h_j}_{(2,\overline{S}_j)}(p)|-3\varepsilon_0\ =\ |\langle
h_j(p),\overline{e}_2^j\rangle|-3\varepsilon_0\\
&=&|\langle h_j(p),e_1^j\rangle+\langle
h_j(p),\overline{e}_2^j-e_1^j\rangle|-3\varepsilon_0\\& \geq& |{
h_j}_{(1,S_j)}(p)|-|\langle h_j(p),\overline{e}_2^j-e_1^j\rangle|-3\varepsilon_0\ >\
|{ h_j}_{(1,S_j)}(p)|-\frac{6\varepsilon_0}{R}|h_j(p)|-3\varepsilon_0,
\end{eqnarray*}
where we have used (c1) and (b1) again.

So, we estimate $|{ h_j}_{(1,S_j)}(p)|$ and $|h_j(p)|$,
\begin{eqnarray*}
|{ h_j}_{(1,S_j)}(p)|&=&|{ h_{j-1}}_{(1,S_j)}(p)+\kappa\,\zeta_j(p)|\ >\
|F_{(1,S_j)}(p)+\kappa\,\zeta_j(p)|-\varepsilon_0\\
&>&|F_{(1,S_j)}(p_j)+\kappa\,\zeta_j(p)|-2\varepsilon_0\ =\
||F_{(1,S_j)}(p_j)|+\kappa\,\zeta_j(p)|-2\varepsilon_0\\
&>&|F(p_j)|-2\varepsilon_0\ >\ r-2\varepsilon_0.
\end{eqnarray*}
Here, we used (b1), (a2) and (b7).

Moreover,
\begin{eqnarray*}
|h_j(p)|&\leq&|h_{j-1}(p)|+\kappa|\zeta_j(p)|+\varepsilon_0\ <\
|F(p)|+3(R-r)+2\varepsilon_0\ <\ 4R-3r+2\varepsilon_0,
\end{eqnarray*}
from (b6), (b1) and (b7).

Therefore,
$$
|G(p)|\ >\
(r-2\varepsilon_0)-\frac{6\varepsilon_0}{R}(4R-3r+2\varepsilon_0)-3\varepsilon_0\ >\
r-\varepsilon_3.
$$
Again, we have initially taken $\varepsilon_0$ small enough in terms of $R,r$ and
$\varepsilon_3$.\\[3mm]
(iv) Let $p\in int(D(p_{j+1},\delta))\cap Q_j(L_0)$. This case is analogous to the
previous one, bearing in mind
$$
|\overline{e}_2^j-e_1^{j+1}|\ \leq\ |\overline{e}_2^j-e_1^{j}|+|e_1^{j}-e_1^{j+1}|\
<\ \frac{6\varepsilon_0}{R}+\frac{\varepsilon_0}{3(R-r)}.
$$
The last inequality is a consequence of (\ref{epsilon}) and (a3).\\[3mm]
(v) Let $p\in Q_j(L_0)\backslash \left(\cup_{k=1}^nint(D(p_k,\delta))\right)$. From
(c1), (b1), (\ref{epsilon}) and (a2), one obtains
\begin{eqnarray*}
|G(p)|&>&|g_j(p)|-\varepsilon_0\ \geq\
|{g_j}_{(2,\overline{S}_j)}(p)|-\varepsilon_0\ =\
|{g_{j-1}}_{(2,\overline{S}_j)}(p)|-\varepsilon_0\\
&>&|{h_n}_{(2,\overline{S}_j)}(p)|-2\varepsilon_0\ >\
|{F}_{(2,\overline{S}_j)}(p)|-3\varepsilon_0\ >\
|{F}_{(1,S_j)}(p)|-3\varepsilon_0-\frac{6\varepsilon_0}{R}\\
&>&|{F}_{(1,S_j)}(p_j)|-4\varepsilon_0-\frac{6\varepsilon_0}{R}\ >\
r-4\varepsilon_0-\frac{6\varepsilon_0}{R}\ >\ r-\varepsilon_3.
\end{eqnarray*}
This finishes Lemma \ref{l4} and also proves Lemma \ref{mainlemma}.
\end{proof}

\footnotesize The first author is partially supported by MCYT-FEDER, Grant No
MTM2007-61775 and the second author by MCYT-FEDER, Grant No MTM2007-65249. The
authors are partially supported by Grupo de Excelencia P06-FQM-01642 Junta de
Andalucía.
\end{document}